\documentclass[11pt, twoside]{article}
\usepackage{amssymb, amsmath, amsthm}
\usepackage[left=25mm, right=25mm, bottom=30mm]{geometry}
\usepackage{inputenc}
\usepackage{fancyhdr}
\usepackage{tikz}
\usetikzlibrary{positioning}
\usepackage{titling}
\usepackage{url}
\usepackage{hyperref}
\predate{}
\postdate{}
\usepackage{lipsum}
\newtheorem{theorem}{Theorem}
\newtheorem{lemma}[theorem]{Lemma}

\newtheorem{claim}{Claim}

\newtheorem{question}[theorem]{Question}

\begin{document}
\newcommand{\Addresses}{{
\bigskip
\footnotesize
\medskip

Maria-Romina~Ivan, \textsc{Department of Pure Mathematics and Mathematical Statistics, Centre for Mathematical Sciences, Wilberforce Road, Cambridge, CB3 0WB, UK.}\par\nopagebreak\textit{Email address: }\texttt{mri25@dpmms.cam.ac.uk}

\medskip

Imre~Leader, \textsc{Department of Pure Mathematics and Mathematical Statistics, Centre for Mathematical Sciences, Wilberforce Road, Cambridge, CB3 0WB, UK.}\par\nopagebreak\textit{Email address: }\texttt{i.leader@dpmms.cam.ac.uk}

\medskip

Mark~Walters, \textsc{School of Mathematical Sciences, Queen Mary University of London, London, E1 4NS, UK.}\par\nopagebreak\textit{Email address: }\texttt{m.walters@qmul.ac.uk}
\medskip}}
\pagestyle{fancy}
\fancyhf{}
\fancyhead [LE, RO] {\thepage}
\fancyhead [CE] {MARIA-ROMINA IVAN, IMRE LEADER AND MARK WALTERS}
\fancyhead [CO] {CONSTRUCTIBLE GRAPHS AND PURSUIT}
\renewcommand{\headrulewidth}{0pt}
\renewcommand{\l}{\rule{6em}{1pt}\ }
\title{\Large\textbf{CONSTRUCTIBLE GRAPHS AND PURSUIT}}
\author{MARIA-ROMINA IVAN, IMRE LEADER AND MARK WALTERS}
\date{}
\maketitle
\begin{abstract}
A (finite or infinite) graph is called constructible if it may be obtained recursively 
from the one-point graph by repeatedly adding dominated vertices. In the finite case, 
the constructible graphs are precisely the cop-win graphs, but for infinite 
graphs the situation is not well understood.

One of our aims in this paper is to give a graph that is cop-win but 
not constructible. This is the first known such example. We also show that every
countable ordinal arises as the rank of some constructible graph, answering a
question of Evron, Solomon and Stahl.
In addition, we give a finite constructible graph for
which there is no construction order whose associated domination map is 
a homomorphism, answering a question of Chastand, Laviolette and Polat.

Lehner showed that every constructible graph is a weak cop win (meaning that the cop
can eventually force the robber out of any finite set). Our other main aim is to
investigate how this notion relates to the notion of `locally constructible' (every finite graph is 
contained in a finite constructible subgraph). We show that, under mild extra conditions, 
every locally constructible graph is a weak cop win.
But we also give an example to show that, in general, a locally constructible 
graph need not be a weak cop win.  
Surprisingly, this graph may even be chosen to be locally finite. 
We also give some open problems. 
    
\end{abstract}

\section{Introduction}

The game of cops and robbers is played on a fixed graph $G$, which for the
moment we will assume is finite. The cop picks a vertex to start at, and the
robber then does the same. Then they move alternately, with the cop moving
first: at each turn the player moves to an adjacent vertex or does not move.
The game is won by the cop if he lands on the robber. We say that $G$ is \textit{cop-win}
if the cop has a winning strategy. Needless to say, if the graph is not connected
then this game is a rather trivial robber win, so we assume from now on that all graphs are
connected.

The cop-win graphs were characterised by Nowakowski and Winkler \cite{NW}. 
It is an easy exercise to see that if the graph contains a dominated vertex,
say $x$, then $G$ is cop-win if and only if $G-x$ is cop-win. (Here as usual we say that a vertex
$y$ \textit{dominates} a vertex $x$ if the set of $x$ and all neighbours of $x$ is contained in the set of
$y$ and all neighbours of $y$.) It is also easy to
see that if no vertex is dominated then the robber has a winning strategy, so
that $G$ is not cop-win -- on each turn, the robber moves to a vertex not adjacent to the cop. 
Putting these together, we see that a finite graph $G$ is
cop-win if and only if it is constructible, meaning that it can be built up from the
one-point graph by repeatedly adding dominated vertices. More precisely, we say that
$G$ is \textit{constructible} if there is an ordering of
its vertices, say $x_1,\ldots,x_n$, such that, for every $k>1$, in the graph $G[x_1,\ldots,x_k]$ the
vertex $x_k$ is dominated by $x_i$ for some $i<k$. We often refer to the given ordering
of the vertices as the \textit{construction ordering}, and the map sending $x_k$ to its
dominating $x_i$ as the \textit{domination map} for this ordering. Note that the construction
ordering, and the domination map for a given ordering, are typically not unique.

We mention briefly that there is also the `reverse' notion of a graph being
\textit{dismantleable}, meaning that we may start with the graph and repeatedly remove
dominated vertices and arrive at the one-point graph. This is of course the same as being constructible
(for finite graphs -- it turns out that in the infinite setting this is
not a useful notion, which is why work on cops and robbers in infinite graphs tends
to focus on concepts to do with constructibility). See the book of Bonato and
Nowakowski \cite{BN} for general background and
a wealth of other results in the finite case, where there are many questions about 
the generalisation where there is more than one cop.

Let us now turn to infinite graphs. The game of cops and robbers has the exact same rules
as before. We remark in passing that if the cop does not have a winning strategy then
the robber has one, for example because the game is an open game (see eg. \cite{AK}).

What about constructibility? The right generalisation of the finite situation is
to allow the vertices to be added recursively, in other words along a
well-ordering. So we say that $G$ is \textit{constructible} if there is an ordinal $\beta$ 
such that its vertices may be
listed as the $x_\alpha$, each $\alpha < \beta$, so that for every $\alpha>0$
the vertex $x_\alpha$ is dominated in the induced graph $G[\{x_\gamma : \gamma \leq \alpha \}]$. 
We say that this well-ordering of the vertices is a \textit{construction order},
with \textit{domination map} as before. The \textit{rank} or \textit{construction time} of $G$ is then the least $\beta$ for which
there is a construction order of order-type $\beta$. If the
rank is $\omega$ then we say that that $G$ is \textit{naturally}
constructible.

It is easy to find examples of constructible graphs that are not cop-win. For
example, a one-way infinite path clearly has this property. However, there is a 
related notion of `weak cop win', introduced by Lehner \cite{L} (after earlier work
by Chastand, Laviolette and Polat \cite{CLP}). A graph $G$ is a called a 
\textit{weak cop win} if there is a strategy for the cop that
guarantees that either the cop catches the robber or the robber has to eventually leave (and never return to)  
every finite set -- in other words, for every vertex the robber only visits
that vertex finitely often. In the usual language of infinite graphs, one could say
that the cop either catches the robber or traps him in one end of the graph (although interestingly, as we will see 
later, this intuition is not really correct). 
For example, the one-way infinite path is a weak cop win.

Lehner \cite{L} gave an elegant argument to show that every constructible graph is a weak
cop win. He asked if the converse also holds. This was answered by 
Evron, Solomon and Stahl \cite{ESS}, who gave 
examples to show that,  interestingly, this need not be the case. But none of those
examples are cop-win, only weak cop-win. One of our aims in this paper is to give an example
of a graph that is actually cop-win and yet is not constructible. We also give a variant of this graph which is a weak cop win, with two ends, but where the robber never has to commit to being in one of these ends. This shows that in some sense the notion of
a weak cop win is more subtle than it might appear.

One of the ingredients of our construction is a finite graph that acts as a kind
of `one-way valve'. This graph, that we call $K$, has the property that the
cop can chase the robber out of it, but `only in one direction'. It is by putting
together copies of $K$ in a certain way that we obtain our desired graph.

This method has an unexpected `spin-off'. In all known examples of finite
constructible graphs, the construction order and domination map could be chosen
in such a way that the domination map was a homomorphism (meaning that if $x$ and
$y$ are adjacent then their images are adjacent or equal). Chastand, Laviolette
and Polat \cite{CLP} asked if this is
always the case. By putting together two copies of $K$ in a certain way, we give a simple counterexample.

Before the paper of Evron, Solomon and Stahl \cite{ESS}, there were no known examples of graphs that were
constructible but not naturally constructible. We stress to the reader how
remarkable this lack of examples was: the problem is that when a graph is constructible there seem to
always be `many' ways to construct it, starting from almost anywhere in the graph,
and this seems to lead to a construction in time $\omega$. This relates to the general reason why cops and robbers on infinite graphs is not so well understood: it seems hard to produce graphs that are cop-win but for an `interesting' (non-trivial) reason, and similarly for weak cop wins. Indeed, Lehner \cite{L} proved that
if $G$ is locally finite and constructible then it must be naturally constructible.
Evron, Solomon and Stahl gave examples of graphs whose rank is greater than $\omega$, and indeed they
showed that the set of ranks of constructible graphs are unbounded (in the
countable ordinals).

They were unable to show that every countable ordinal arises as a rank, and they
asked whether or not this holds. We show that this is indeed
the case: our starting-point is again based on building up a graph from copies
of $K$.

Another part of our paper is concerned with a weakening of the notion of
constructibility to `local constructibility'. Returning to weak cop wins,
one would naturally imagine that the following
generalisation of Lehner's result holds: any graph that is \textit{locally constructible}
(meaning that every finite set is contained in a finite constructible set) should
be a weak cop win. This should especially be true in the locally finite case.
The intuition is that the cop can force the robber out of any finite set using
the finite constructible superset of that finite set -- perhaps with some compactness argument to
make these strategies `consistent' over different finite sets. We
mention in passing that the notion of `locally constructible' is somehow
more tangible that that of `constructible'. For example, it is clear
that we can test whether or not a countable graph is locally
constructible in time $\omega$, whereas we see no way to test for
constructibility even in any (countable) ordinal time.

We are able to 
prove this generalisation under a small strengthening of local constructibility: any graph that is `consistently' locally
constructible (which we define below) is indeed a weak cop win. Remarkably, though, some such condition is
indeed necessary: our final example is a locally constructible graph
that is not a weak cop win. In fact, this graph can even be taken to be locally finite as well. These are by far the most delicate and involved constructions in the paper.

The plan of the paper is as follows. In Section~2 we introduce the
graph $K$, and as a `warm-up' we use this to build a finite graph
that is constructible but for which no construction order has a 
domination map that is a homomorphism. Although this result is not one
of the main ones of the paper, we prove it here so as to get the reader
used to the properties of $K$. In Section~3 we exhibit a graph that is
cop-win but not constructible. Then in Section~4 we turn to locally
constructible graphs, showing that a consistently locally constructible
graph (whether or not locally finite) must be weak cop-win. Section~5 contains our
examples of locally constructible graphs that are not weak cop-win.
We return to general constructibility in Section~6, where we find
graphs whose ranks are any given countable ordinal. 
We conclude in Section~7 with several open problems.

For general background on cops and robbers, see \cite{BN}. For results
particularly dealing with infinite graphs, see (apart from the papers mentioned above)
Bonato, Golovach, Hahn and Kratochvil \cite{BGHK} for results about capture times, 
Polat \cite{Polat1}\cite{Polat2}\cite{Polat3} for material about dismantleability and related aspects, and 
Hahn, Laviolette, Sauer and Woodrow \cite{HLSW} for other structural properties.
For some very attractive results on
the computability aspects of constructibility and pursuit see Stahl \cite{S}.

Our notation is standard. Our graphs are undirected and loopless. For a subset $U$ of the vertices of a graph $G$, we write $G[U]$ for the graph induced by $U$. For two vertices $x$ and 
$y$ we sometimes write $x \sim y$ if $x$ and $y$ are either adjacent or equal. 
We often talk informally about vertices being `added' in a construction order, or `removed' for dismantling. The 
chosen vertex that dominates a vertex $x$ in a construction order (in other words, the image of $x$ under the domination
map) is often referred to as the `parent' of $x$. Finally, for a 
constructible graph with given construction order and given domination map $\delta$, the \textit{trail} of a vertex $x$ is the (necessarily finite) sequence $x, \delta(x), 
\delta(\delta(x)), \ldots$ that starts at $x$ and terminates at the root
(the initial vertex) of the construction order.

\section{The graph $K$ and and a finite application}

\par In this section we introduce a finite constructible graph that is going to be pivotal for our later constructions. We call this graph $K$, pictured below. Note that $x$ is the unique dominated vertex and $y$ is its unique dominating vertex. In particular, in any construction ordering the vertex $x$ must come last.
\begin{figure}[h]
    \centering
    \includegraphics[width=10cm]{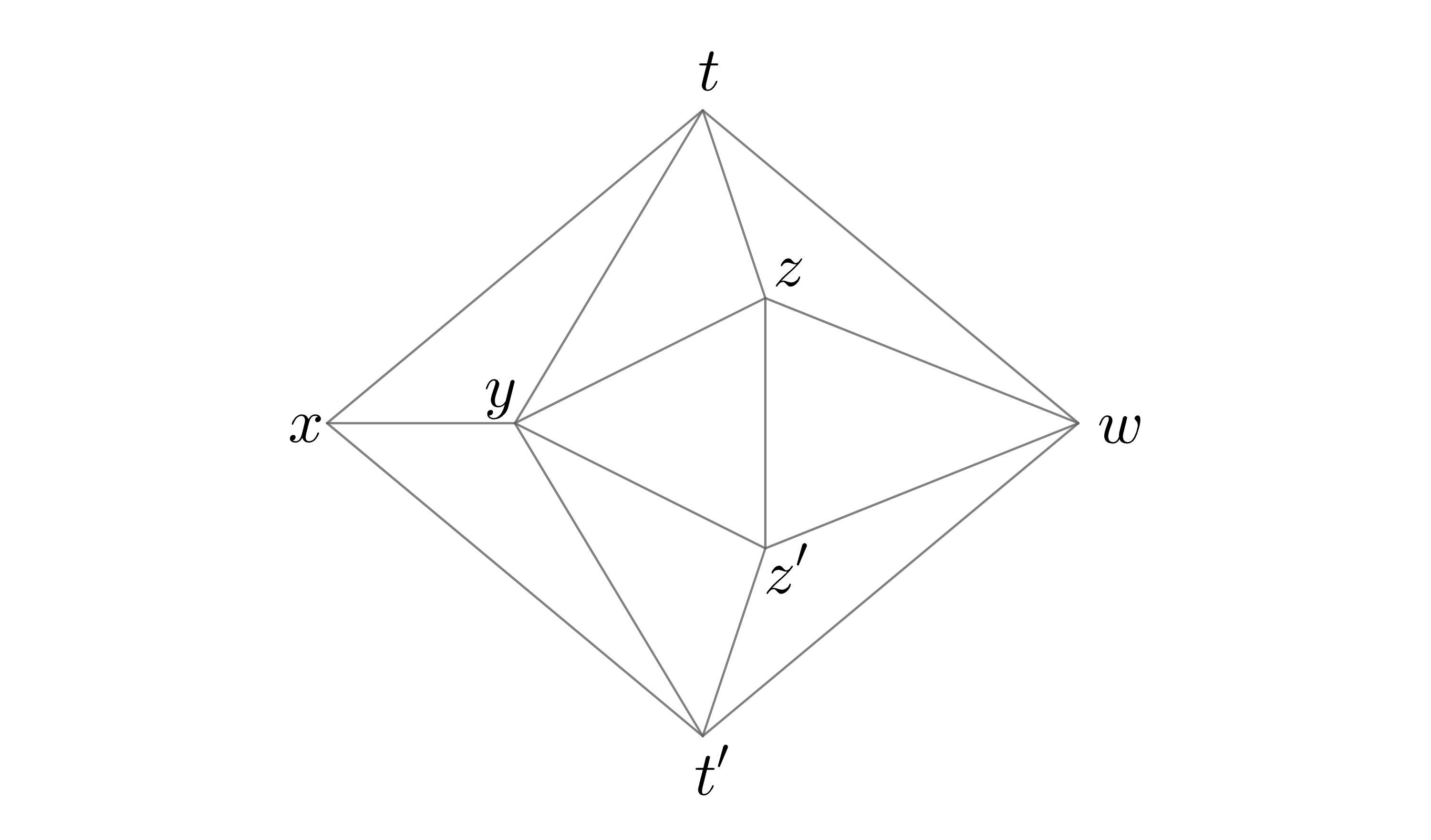}
    \caption{The graph $K$.}
    \label{fig:graph_K}
\end{figure}
To see that $K$ is constructible, or equivalently dismantleable, we observe that the vertex $x$ is dominated by $y$. Once $x$ is removed $t$ and $t'$ are dominated by $z$ and $z'$ respectively. Once they are removed, $z$ is joined to everything so it dominates all remaining vertices. Thus the graph is dismantleable.\par We note that from any vertex in the graph the robber can guarantee to either get to $x$ without being caught or to survive forever. For example, if the robber is at $w$, he waits until the cop comes to one of
$t,z,t',z'$. If the cop is at $t$ or $z$ the robber goes to $t'$, and if the cop is at $t'$ or $z'$ the robber goes to $t$. After that he either stays at $t$, goes back to $w$ or goes to $x$. (Alternatively, as the robber can obviously avoid being caught whenever he is
at a non-dominated vertex, it follows that he can only be caught at $x$.)
\par The following lemma is one of the main results about $K$ which we use in our constructions.
\begin{lemma} Let $G$ be a constructible graph that has $K$ as an induced subgraph. Moreover, let all the edges between $K$ and $G\setminus K$ have their $K$-end at one of $x$ and $y$. Then, in any construction order, $x$ must be the last vertex of $K$ added, and its parent must be $y$.  \label{help}\end{lemma}
\begin{proof}Suppose that $v\not=x$ is the last vertex of $K$ added. Then $v$, at this point in the construction, must be dominated by some vertex already added, either in $K$ or the part of $G$ so far
constructed. However, since $v\neq x$, $v$ is not dominated by any vertex in $K$, and $v$ has neighbours in $K\setminus \{x,y\}$, so it is not dominated by any vertex outside $K$. 

Therefore the last vertex of $K$ added must be $x$. Since $x$ and $t$ are adjacent, the parent of $x$ cannot be outside $K$, and so its parent must be
$y$.
\end{proof}

To see how these properties of $K$ may be used, we give a simple example of a (finite) constructible graph in which the domination order
cannot be chosen to be a homomorphism.

\begin{theorem}\label{t:no-hom} There exists a finite constructible graph for which no domination map is a homomorphism.
\end{theorem}
\begin{proof} We construct the graph $G$ by taking two disjoint copies of $K$, $K_1$ and $K_2$, and identifying the $x$ of the first with the $y$ of the second. The graph is pictured below.
\begin{figure}[h]
    \centering
    \includegraphics[width=10 cm]{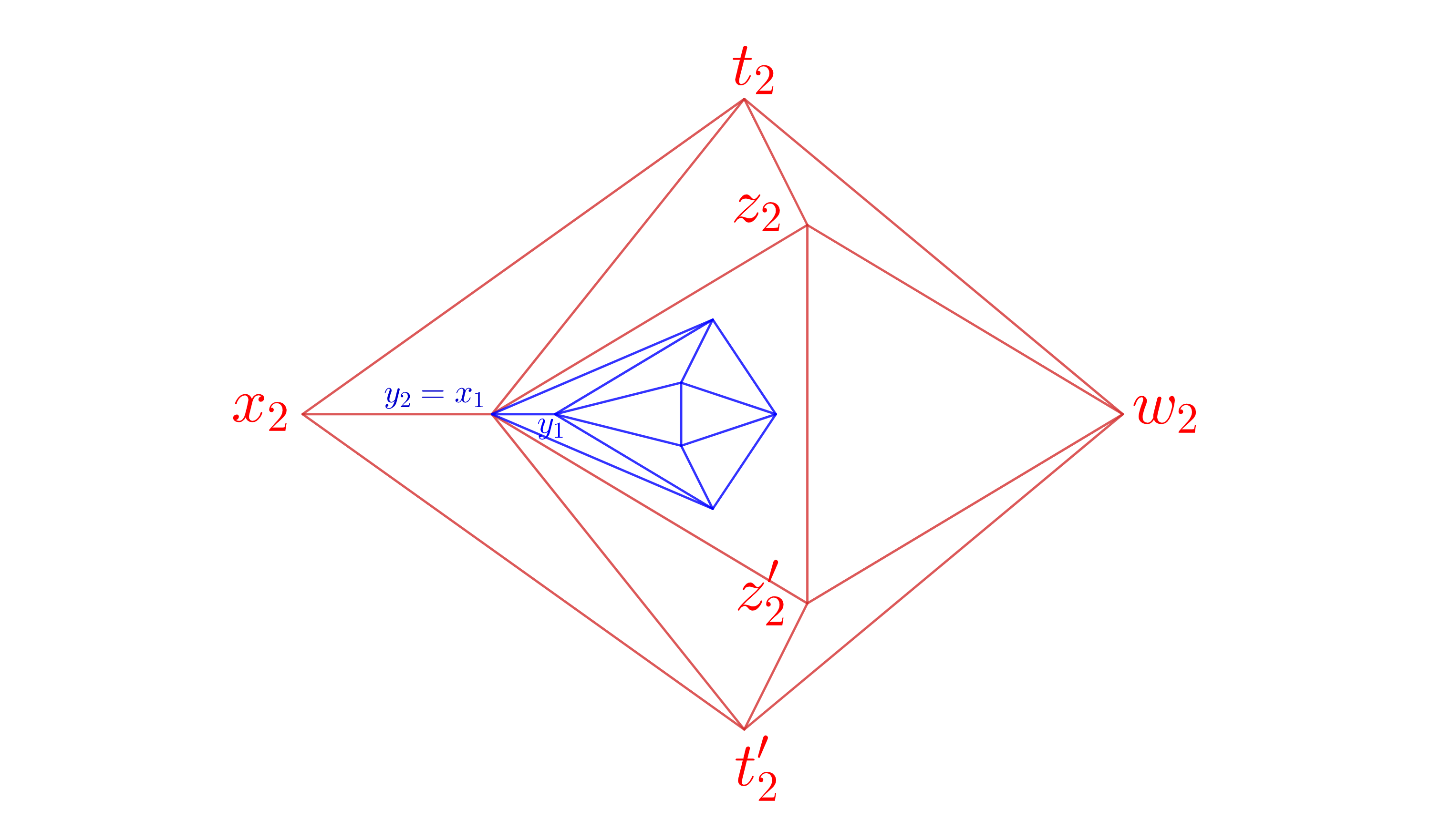}
    \caption{The graph $G$ from Theorem~\ref{t:no-hom} showing the two copies of $K$, one in red and one in blue.}
    \label{fig:no-hom}
\end{figure}

\par First of all, the above graph is constructible. To see this, we prove that it is dismantleable. We can first remove $x_2$ as it is dominated by $x_1=y_2$, then $t_2$ and $t'_2$ as they are dominated by $z_2$ and $z'_2$ respectively. Now we can remove $w_2$ (dominated by $z_2$) and then $z_2$ and $z'_2$. Now we are left with $K_1$ which we know is dismantleable. This finishes the proof that $G$ is constructible.
\par We now show that regardless of construction, the domination map is not a homomorphism. In other words, for any construction order of $G$, there exist two adjacent vertices in $G$ such that their parents cannot be chosen to be adjacent or equal.\par Note first that $x_1=y_2$ must have parent $y_1$: by Lemma \ref{help}, $x_1$ has to be the last vertex added in $K_1$, which implies that its parent must be $y_1$.\par If the domination map were a homomorphism, then all the neighbours of $x_1=y_2$ in $K_2$ would have to have parents that are adjacent to (or equal to) $y_1$, and the only 
possible vertex for this is $y_2$. However, if the vertices $t_2$, $t'_2$, $z_2$ and $z'_2$ all have parent $y_2$, then we cannot construct 
$w_2$: it has to be constructed before the last of these four neighbours is constructed, but all these four neighbors are adjacent to $w_2$, while $y_2$ is not. This shows that, no matter what the construction order is, the domination map cannot be a homomorphism.\end{proof}
\section{A non-constructible cop-win graph}
\par In this section we show that there exists a non-constructible
graph on which the cop can always win, meaning as before that he can always
catch the robber in finite time.
\par We
begin with an infinite sequence of copies of $K$,
$K_1, K_2, \ldots$, where we identify $y_i$ with
$x_{i+1}$. Finally we add a new vertex which we call $0$ and join it
to all $x_i$ and $y_i$. We call this graph, which is pictured below,
$G$. Note that the line of copies of $K$ extends `to the right and not to the left': this will be crucial.
\begin{figure}[h]
    \centering
    \includegraphics[width=10cm]{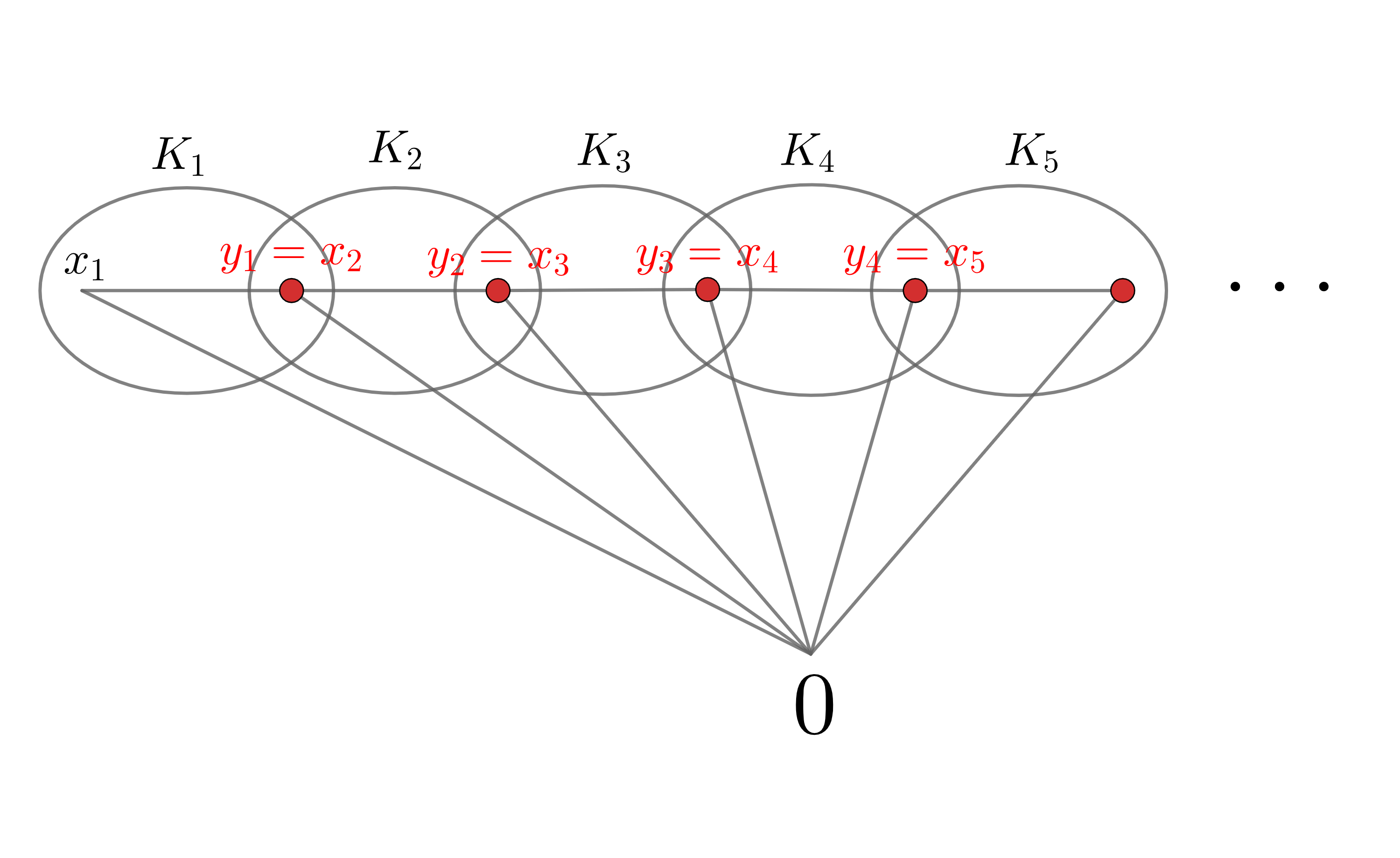}
    \caption{The graph $G$ for Theorem~\ref{t:cop-win}.}
    \label{fig:f:cop-win}
\end{figure}
\begin{theorem}\label{t:cop-win} The graph $G$ is cop-win, but is not constructible.
\end{theorem}
\begin{proof} To show that the graph is not constructible, suppose for a contradiction that we have a construction order for it. 

Now, by Lemma\ref{help} we have that the parent of $x_1$ must be $y_1$. But $y_1=x_2$, and
again by Lemma~\ref{help} the parent of $x_2$ must be $y_2$. So in fact the parent of
$x_i$ is $x_{i+1}$ for all $i$, and this contradicts the fact that the 
construction order is a well-order.

We are left to show that $G$ is a cop-win graph. Whatever the cop's initial position, he can move in at most 2 steps to $0$ (or indeed he may just start at $0$). After the cop reaches $0$, the robber makes his move, and now the robber must be inside some $K_i$. If the robber is at $x_i$ or $y_i$ then he is caught immediately as these vertices are adjacent to $0$. So assume the robber is at some other vertex in $K_i$. Note that the only vertices in $K_i$ with any neighbour outside $K_i$ are $x_i$ and $y_i$, so the robber cannot leave $K_i$ until he reaches one of those two vertices. 
  
Now the cop moves to $y_i$, which will be the start of a chain of forced moves for the robber. Since $y_i$ is adjacent to all vertices in $K_i$ except $w_i$, the robber has to move to $w_i$. Next the cop moves to $z_i$, forcing the robber to $t'_i$. 
Then the cop moves to $z_i'$ forcing the robber to $x_i$. Finally the cop moves to $y_i$, forcing the robber to leave $K_i$  -- and not go to $0$ since $y_i$ and $0$ are adjacent. 
The robber must thus move into $K_{i-1}$, and the cop follows him by moving to $x_i=y_{i-1}$, so that the process can repeat in $K_{i-1}$.

Continuing in this way, the cop forces the robber out of each copy of $K$ in turn until the robber reaches $x_0$, where he cannot avoid getting caught.
\end{proof}

There are some variants of the above graph that have some interesting properties. For example, if we remove the vertex $0$ then we have a rather simple example of a non-constructible graph that is not cop-win, but is weak cop-win. Indeed, if the cop is in a block to the right of the robber, then the cop can force the robber out of each block in turn as we have see above, and the robber gets caught. However, if the cop is on the robber's left, then the cop runs to the right and the only way the robber can avoid being to the left of the cop at some point is by also running to the right.

In terms of ends of graphs (see \cite{RD} for general background), it is natural to assume that in a weak cop win graph the cop can `force a robber into one end', in the sense that the set of possible ends to which the robber's eventual path can belong, after say time $n$, should shrink down to one end as $n$ tends to infinity. But, surprisingly, this
is not the case. Indeed, consider the variant of the above graph $G$ in which we have a two-way infinite line of copies of $K$. This graph has two ends.
But when the cop chases the robber off to the right, then at every time the robber is always free to `change direction' by going past the cop (without being caught) and then
running off to the left. So the set of possible ends remains both ends of the graph, for all time.  

\section{Locally constructible graphs}

We have seen that there are non-constructible graphs that are weak cop wins and even an example that is a cop win. In this section we introduce a weaker notion that captures many of the key properties of constructibility. 

Recall that we call a graph $G$ locally constructible if, for any finite set of vertices $V$, there is a finite set of vertices $U$ containing $V$ such that $G[U]$ is constructible. (We remark that actually one could omit the condition that $U$ is
finite, since if $U$ is infinite then the union of the trails in $U$ of all vertices
in $V$ is a finite constructible graph.)

One motivation for this definition is that it easily implies that the cop can force the robber to leave any finite set of vertices (although the robber may return later). Indeed, given a finite set of vertices $V$, take $U$ as in the definition of locally constructible. If the robber stays on $V$ then he necessarily stays on $U$, and since $G[U]$ is constructible, the standard finite result shows that the cop catches him.

However, this does not show that the game is a weak cop win as nothing in the above argument prevents the robber from returning to the finite set at some later stage. One may naturally feel that some form of compactness argument would yield, at least for locally finite graphs, some way of combining the local strategies from these local constructions into a global strategy. Rather surprisingly, as we shall see in the next section, this is not the case. 

First, though, we prove that the cop does have a weak winning strategy in the locally constructible case with an extra condition, which is that the notion of parent is consistent. More precisely, we say a graph $G$ is \emph{consistently} locally constructible if there is a nested sequence of finite induced subgraphs $G_i$ with $\bigcup_i G_i=G$, vertices $v_i\in G_i$, and maps $\delta_i\colon G_i\setminus\{v_i\}\to G_i$ such that:

\begin{enumerate}
    \item Each $G_i$ is constructible with domination map given by $\delta_i$ and root $v_i$; 
    \item The maps are consistent: if $v\in (G_i\setminus\{v_i\})\cap (G_j \setminus\{v_j\})$ then $\delta_i(v)=\delta_j(v)$.
\end{enumerate}
Note that we do not require that the construction orders are consistent, just that the notion of parent is.

We remark that in our example of a graph that is a weak cop win but not constructible above, the graph is consistently locally constructible in a very natural way: just take any finite block of the $K$s and construct it starting from the right.  

We will need the following finitary result of Isler, Kannan and Khanna \cite{IKK}. We provide a short proof for the reader's convenience. 

\begin{lemma}
Let $G$ be a finite constructible graph. Consider the following cop strategy: he starts at the root, and then on each turn he moves to the maximal vertex on the trail of the robber's current position that he is adjacent to. Then 
\begin{enumerate}
    \item This strategy is well defined: there is always a neighbour of the cop's current position that is on the trail of the robber.
    \item If the robber is at some vertex $v$ and returns there at some later time, then the cop is strictly closer to the robber on the second occasion.
\end{enumerate}
In particular, this strategy is winning for the cop.
\end{lemma}
\begin{proof}
We start by fixing a construction ordering $<$ and associated domination map $\delta$.

Suppose that $u$ and $v$ are any two vertices which are joined in $G$. We claim that any vertex $u'$ on the trail of $u$ is joined to the maximal vertex $v'$ on the trail of $v$ with $v'\le u'$, and vice versa. We prove this by reverse induction on the set of vertices in the union of the trails of $u$ and $v$.

Suppose that it holds for some vertex $u'$ in the trail of $u$: thus $u'$ is joined to $v'$ where $v'$ is the maximal vertex in the trail of $v$ with $v'\le u'$. It is immediate from the definition of domination and the domination map that $v'$ is 
adjacent to (or equal to) $u''=\delta(u')$. Now, $v'$ is the greatest vertex in the trail of $v$ which is at most $u'$, and $u''$ is the greatest vertex in the trail of $u$ that is less than $u'$. Therefore one of $v'$ and $u''$ is the next-largest vertex in the union of the trails, and the other is the greatest vertex less than that in the other trail. In either case we have the inductive step, which establishes our claim.

Now suppose the robber is at $x$, the cop is at $x'$ on the trail of $x$, and the robber moves to $y$. We see that $x'$ is joined to the greatest vertex $v$ on the trail of $y$ with $v\le x'$. In particular, $x'$ is joined to a vertex on the trail of $y$, and therefore the strategy is well defined. 

For the second part, suppose that $x'=\delta^k(x)$, and the cop moves to $y'=\delta^\ell(y)$. While the cop's position may decrease (i.e., we may have $y'<x'$) we claim that the position one step closer to the robber on the trail does increase, i.e., $\delta^{\ell-1}(y)>\delta^{k-1}(x)$. Indeed, let $y''=\delta^{\ell-1}(y)$. Applying the above to $x'$ we see that $y'$ is at least the greatest vertex on the trail of $y$ which is at most $x'$, and thus $y''> x'$. Now applying the above to $y''$ we see that $y''$ is joined to the greatest vertex $v$ on the trail of $x$ with $v\le y''$. Since the cop did not move to $y''$ we know that $x'$ is not joined to $y''$; in particular, $v\not=x'$. Since $y''>x'$, we see that $v\ge x'$. Combining these, we have $v> x'$, so $v\ge x''$. Thus $y''\ge v\ge x''$, as required. 

Finally, note that these two conditions, together with the trivial observation that the root is a common ancestor of the whole graph, imply that the graph is a weak cop win: each time the robber returns to a vertex the cop is strictly closer, and each vertex only has finitely many ancestors.
\end{proof}

We now prove the main result of this section.

\begin{theorem}
Let $G$ be a consistently locally constructible graph. Then $G$ is a weak cop win.
\end{theorem}
\begin{proof}
Define the `domination' map $\delta:G\to G$ by $\delta(v)=\delta_i(v)$ for any of the local domination maps that occur in the definition of consistently locally constructible that are defined at $v$. In particular, this means we can talk about the `trail' of any vertex (although now there is no reason why the trail should be finite).

The cop strategy is as follows. Suppose that the robber is at $x$ and the cop at $z$.
\begin{itemize}
    \item Case 1. There exists a neighbour of the cop's current position that is on the trail of the robber. In this case, the cop moves to the most recent ancestor of the robber that he can reach: that is, he moves to the vertex $z'=\delta^k(x)$ with minimal $k$ with $z\sim z'$.
    \item Case 2. Otherwise, the cop moves to the parent of his current position: that is, he moves to $\delta(z)$.
\end{itemize}
Suppose that the cop is ever in Case 1. Then we claim that he remains in Case 1. Indeed, after the cop's move the cop is at $z'$ on the trail of the robber at $x$. The robber moves to some vertex $x'$. If we take any $G_i$  containing all of $x,x',z'$, then Lemma 4 tells us that there is a neighbour of $z'$ on the trail of $x'$ in $G_i$. Since trails in $G_i$ are the same as trails in $G$, the claim follows.

Furthermore, if Case 1 ever occurs then the robber can only visit any vertex finitely many times. Indeed, suppose that the robber has a sequence of moves $x_1,x_2,\ldots x_n,x_1$ starting and finishing at $x_1$, and the cop's sequence under this strategy is $y_1,y_2,\ldots,y_n,y_{n+1}$. Let $G_i$ be chosen to contain all the $x_i$ and $y_i$.  Since the parent maps are consistent, we see that the cop is following exactly the winning strategy in $G_i$, and so in particular, by Lemma 4, $y_{n+1}$ is a (strictly) more recent ancestor of $x$ than $y_1$ was. 

Hence, each time the robber return to $x_1$, the cop is closer on the robber's trail, and after some finite number of loops he catches the robber.

To complete the proof, it remains to show that Case 1 must occur at some point. We note that any two vertices have a common ancestor under $\delta$ (just take the root vertex in any $G_i$ containing the two vertices). Hence, as the cop follows Case 2 by moving to the parent vertex, every vertex becomes a descendant of the cop's position at some time.  
\end{proof}

It is clear that, while the consistency condition makes this proof 
work, it is not the `right' condition. Indeed, even our example of a cop win that is not constructible is actually not consistently constructible, 
since in any subgraph containing the root the rightmost vertex in the 
$K$s, and only that vertex, has the root as its parent.  

\section{Locally constructible graphs may not be weak cop-win}
In this section we first exhibit a locally constructible graph that is not weak cop-win. Our graph is not locally finite, but by carefully modifying the way it is built up we are able to find a locally finite locally constructible graph that is not weak cop-win. The lemma below is at the heart of our construction: it allows us to pass from any graph to a constructible one.

We write $P_n$ for the path of length $n$, and view its vertices as
$0,1,\cdots,n$. For a finite graph $G$ and a positive integer $n$, 
we write $G*P_n$ for the graph with vertex set $G\times P_n$ in which
$(x,j)$ is joined to $(x',j')$ if either $x\sim x'$ and $j\sim j'$, or $j=j'=n$.

\begin{lemma}
For any finite graph $G$, the graph $G*P_n$ is constructible. Moreover, if $G$ is not constructible, then in $G*P_n$ the cop can be forced to visit a vertex of the form $(x,n)$ for some $x$ before catching the robber.
\label{construction}
\end{lemma}
\begin{proof}
Note that the vertices with second coordinate $n$ form a complete graph, and so can be constructed first. Once we have these vertices, we observe that the vertex $(x, n-1)$ is dominated by the vertex $(x,n)$, and so we can now add
all of the vertices with second coordinate $n-1$. Continuing in this way, we can add all the vertices, and thus the graph is constructible.

Now suppose that the graph $G$ is not constructible. This means that the robber can avoid being caught on $G$. Thus, if the cop never visits a vertex with second coordinate $n$, we can pretend by projection that the chase happens on $G$, so that the robber can avoid being caught.
We conclude that the cop is be forced to visit a vertex $(x, n)$, for some $x$, before catching the robber.
\end{proof}
The next step is to observe that if we have a graph $G$, and we attach disjoint 4-cycles to all of its vertices, the robber will always win in this new graph regardless of the starting position, by staying on the 4-cycle associated with his starting vertex.\par More precisely, 
let $C_4$ be a  4-cycle, say on vertices $0,1,2,3$, and let $G$ be any graph. We define the graph $G \cdot C_4$ by joining 
$(x,y)$ to $(x',y')$ if either $x=x'$ and  $y$ is adjacent to $y'$, or $x$ is adjacent to $x'$ and $y=y'=0$. As explained above, this graph is clearly a robber win regardless of the starting position.\par Therefore, if we start with the graph $C_4$, which is not constructible, and look at $C_4, C_4 * P_n, (C_4 * P_n) \cdot C_4, 
((C_4 * P_n) \cdot C_4) *P_n, \ldots$, then we are alternating between constructible and non-constructible graphs. To achieve locally constructibility without being a weak cop win, it is reasonable to take the union of these graphs. The intuition behind this is that, although the cop can win on each of the constructible stages, namely the ones after taking a product with $P_n$, he has to go a long way from the robber, as shown in Lemma \ref{construction}. This gives the robber time to get back to the origin and then head off into an extra coordinate.

Now we make this idea precise. The reader should bear in mind that the graph
$\mathcal G$ constructed below is precisely the `nested union' of the above sequence of
graphs. 
\begin{theorem}\label{lcnlf}
There exists a graph $\mathcal G$ which is locally constructible, but is not a weak cop win.
\end{theorem}
\begin{proof} We define the graph $\mathcal G$ as follows. The vertex set is $C_4\times P_{6}\times C_4\times P_{6}\times C_4\times
P_{6}\times\ldots$, where we insist that all but finitely many of the
coordinates are 0. Let $\widehat 0$ be the vertex where all coordinates are 0. In order to define the edges we consider three cases. Let $v$ and $v'$ be two vertices.\par If all their $C_4$ coordinates are $0$ and there is no $P_6$ coordinate in which both vertices are 6, then $v$ is adjacent to $v'$ if and only if they differ by at most 1 in all $P_6$ coordinates.
\par Otherwise let $m_1$ be the maximal $C_4$ coordinate in which $v$ and $v'$ are not both $0$, and $m_2$ the maximal $P_6$ coordinate in which both $v$ and $v'$ 
are 6 (and we set $m_i=0$ if the corresponding coordinate does not exist).
\par If $m_1<m_2$, then $v$ is adjacent to $v'$ if and only if, after the $m_2^{\text{th}}$ coordinate, all their $P_6$ coordinates differ by at most 1 -- note that after the $m_2^{\text{th}}$ coordinates all their $C_4$ coordinates are 0 by definition. \par 
If $m_1>m_2$, then $v$ is adjacent to $v'$ if and only if they agree on all coordinates less than $m_1$, differ by at most 1 in the $m_1^{\text{th}}$ coordinate, and differ by at most 1 in all the $P_6$ coordinates greater that $m_1$. 
\begin{claim} The graph $\mathcal G$ is locally constructible.
\end{claim}
\begin{proof} We observe that the graph we get if we restrict to all vertices which are always zero after some particular $P_6$ coordinate is finite and, by Lemma \ref{construction}, is constructible. Another way to see that this graph is constructible is to show that the cop wins on this graph. Indeed, the cop goes to level 6 (the maximum level) in the final $P_6$ coordinate. Let that coordinate be $m$. He is then able to immediately move to a vertex above the robber on the rest of the coordinates. Then, after each robber move, if the robber is at the same level as, or one below, the cop, then the cop immediately catches him. Otherwise the cop moves to stay above the robber on the rest of the coordinates, while reducing the $m^\text{th}$ coordinate by 1. In this way the cop must catch the robber by the time the cop reaches level 0.
\end{proof}
\begin{claim} The graph $\mathcal G$ is not weak cop-win.
\end{claim}
\begin{proof}
The robber's strategy is to always have all coordinates zero with at most one exception, and that exception is in a cycle coordinate. It is clear that after the cop chooses his starting position, the robber can choose a large cycle coordinate $m_0$ and start at the vertex with 2 in the $m_0$ coordinate and 0 elsewhere. Note that this implies that the robber is distance at least 2 from the cop. The robber commits to stay in this cycle (that is, all coordinates except the $m_0^\text{th}$ coordinate are zero) until he reaches $\widehat{0}$, after which he enters a different cycle and the whole process repeats. 

In the discussion that follows the cop always moves first. We define 3 stages of the strategy, where $m$ is the cycle coordinate the robber is currently committed to stay in before he gets to $\widehat{0}$. 
\\

\noindent\textit{Stage 1.} The robber is not at $\widehat{0}$ and the cop's vertex has no path coordinate $6$. Furthermore,  either the $m^{\text{th}}$ coordinate of the cop's vertex is 2 different from the $m^{\text{th}}$ coordinate of the robber's vertex, or it is 1 different and the cop's vertex has a non-zero earlier coordinate.\\
\textit{Stage 2.} The robber is not at $\widehat 0$ and the cop's vertex has a 6 in some path coordinate.\\
\textit{Stage 3.} The robber is at $\widehat 0$ and the cop is at least distance 2 away from the robber.
\\
\par Suppose we are in Stage 1 of the strategy, the cop is at $v$ and the robber is at $w$. By the definition of the edges of $\mathcal{G}$, in one move the cop can go to a vertex $v'$ that differs from $v$ either in the  $m^\text{th}$ coordinate or in some coordinate less than $m$ -- these two cases are disjoint by construction. In either case the coordinates of $v'$ greater than $m$ may differ from those of $v$. If $v'$ has a 6 in some $P_6$ coordinate then the robber does not move and we are now in Stage 2. Thus assume $v'$ does not have a 6 in any $P_6$ coordinate.

If the vertex $v'$  differs from $v$ in some coordinates greater than $m$, then they have the same $m^\text{th}$ coordinate and, in particular, $v'$ and $w$ differ in the $m^\text{th}$ coordinate by at least 1. Thus the robber moves (if necessary) to a vertex $w'$ that differs from $v'$ by at least 2 in the $m^\text{th}$ coordinate. In this case we are either in Stage 1 or, if the robber has reached $\widehat{0}$, in Stage 3.

Finally, if $v$ and $v'$ differ in the $m^\text{th}$ coordinate, then the robber moves to a vertex $w'$ such that the difference between the $m^\text{th}$ coordinates of $v'$ and $w'$ is the same as the difference between the $m^{\text{th}}$ coordinates of $v$ and $w$. Again, we are either in Stage 1 or, if the robber has reached $\widehat{0}$, in  Stage 3.
\par If we are in Stage 2 of the strategy, we observe that the cop is distance at least 6 from $\widehat{0}$. Indeed, let the cop be at $v_0$, and fix a minimal path from $v_0$ to $\widehat{0}$. Consider the maximal $P_6$ coordinate that is ever 6 on this path. This coordinate needs to become 0 and can only change by at most 1 at each step along the path. Thus the path has length at least 6.

It follows that the robber can reach $\widehat{0}$, which takes at most 2 steps as his vertex has all coordinates 0 except for one cycle coordinate, and we are now in Stage 3.\par Finally, suppose we are in Stage 3 and the cop moves somewhere. He must still be at least distance 1 from $\widehat{0}$. The robber picks a new cycle coordinate $m'$ where $m'$ is greater than any of the non-zero coordinates of the cop's vertex, and moves to 1 in this cycle coordinate. We are now back to Stage 1.\par This strategy ensures the robber is never caught. Moreover, the robber either stays in Stage 1 after some point, which means he stays on one particular cycle forever, or he reaches Stage 3 infinitely often, so visiting $\widehat{0}$ infinitely often. We conclude that this graph is not a weak cop win.
\end{proof}
This concludes the proof of Theorem~\ref{lcnlf}.
\end{proof}
\par The above construction gives us a locally constructible graph that is not weak cop-win. However, this graph is not locally finite, as for example the degree of $\widehat{0}$ is infinite. It is natural to ask what happens if we insist that the graph is locally finite. Does this, together with the condition that it is locally constructible, guarantee that the graph is weak cop-win? Below we answer this question negatively by modifying the previous construction so that the graph is locally finite and yet not weak cop-win.

The key extra idea is to obtain locally finiteness by attaching the (iterated) graphs $G*P_6$ from the previous construction along the vertices of an infinite path rather than all to the same vertex. However, this means that it takes the robber longer and longer to return to the origin, so rather than using $G*P_6$ each time we will have to use a more involved construction, and in particular we will need to use an increasing sequence of paths lengths rather than always using $P_6$ when constructing the graphs. 

First we make precise what we mean by the description above of `attaching graphs along the vertices  of an infinite path'. Let $(G_n)_{n\geq 0}$ be any nested sequence of finite graphs -- in other words $G_n$ is a fixed induced subgraph of $G_m$ for all $m \geq n$. We define the \textit{union graph} $\bigsqcup G_n$ to be the graph with vertex set the disjoint union of the vertex sets of all $G_n$, which we view as pairs $(n,x)$ where $n\in\mathbb N$ and $x\in G_n$, with $(n,x)$ adjacent to $(n',x')$ if $|n-n'|\leq 1$ and $x\sim x'$. 

We observe that if a particular $G_k$ is constructible then the subgraph of $\bigsqcup G_n$ given by the vertices $(m,x)$ with $m\leq k$ is constructible. Indeed, we first construct the graph with vertices $(k,x)$ which, because it is isomorphic to $G_k$, is constructible. As before, each vertex $(k-1,x)$ is now dominated by $(k,x)$, and so we can add the entire $k-1$ layer. Continuing in this way we add all the layers, and so the graph is constructible.

Next we define an important step in our construction of each of the graphs $G_n$. This is analogous to Lemma~\ref{construction}, but modified to our new setting. Let $G$ be a finite graph. We say that $G'$ is the \textit{hive graph} of $G$ of \textit{height} $n$ if $G'$ has vertex set $G\times\{0,1,\cdots,n\}$ together with a special vertex $v$ called the \textit{hive vertex} that is adjacent to all vertices of the form $(x,n)$, and $(x, i)$ is adjacent to $(x',i')$ if $x\sim x'$ and $|i-i'|\leq 1$. 

The key points of the hive construction are that, for any $G$, the graph $G'$ is constructible (just start from the hive vertex, then construct layer $n$, then layer $n-1$, and so on down to layer 0 in turn), and that if $G$ is not constructible then the cop cannot win without visiting the hive vertex --  which is a long way from the 0-layer. 

We are now in a position to define our example $\mathcal{H}$ of a locally finite locally constructible graph that is not a weak cop win. We start with $G_0$ as a single vertex $0$. Given $G_{n-1}$, we form $H_n$ by adding a new copy of $C_4$ at $0$ (in other words, we take the disjoint union of our graph with $C_4$ and then identify the two vertices called $0$). We then set $G_n$ to be the hive graph of $H_n$ of height $l_n=2n+5$ with hive vertex $v_n$. The graphs $G_n$ are naturally nested with $G_{n-1}$ a subset of $H_n$ which in turn is a subset of the 0-level of $G_n$. Finally, we define $\mathcal H$ to be the union graph $\bigsqcup G_n$.  We call the vertex $(0,0)$ the \textit{origin} and the set $S=\{(n,0):n\in\mathbb N\}$ the \textit{spine}.

Figure~\ref{fig:hive_graph} below shows how the graph $G_2$ is built up (but with $l_1=l_2=3$ for 
readability). We start with $G_0$, which is the single purple vertex. Next we form $H_1$ by adding the blue 4-cycle. From $H_1$ we form the red hive graph $G_1$ with hive vertex $v_1$ and height 3. We then form $H_2$ by attaching the green 4-cycle to the origin (the purple vertex). Finally, we form $G_2$, the hive graph of $H_2$ with height 3 and hive vertex $v_2$. The dotted lines are drawn to indicate that there are edges between the 4-cycles, between the copies of $H_2$, and so on. 
\begin{figure}[h]
\centering
\includegraphics[width=12cm]{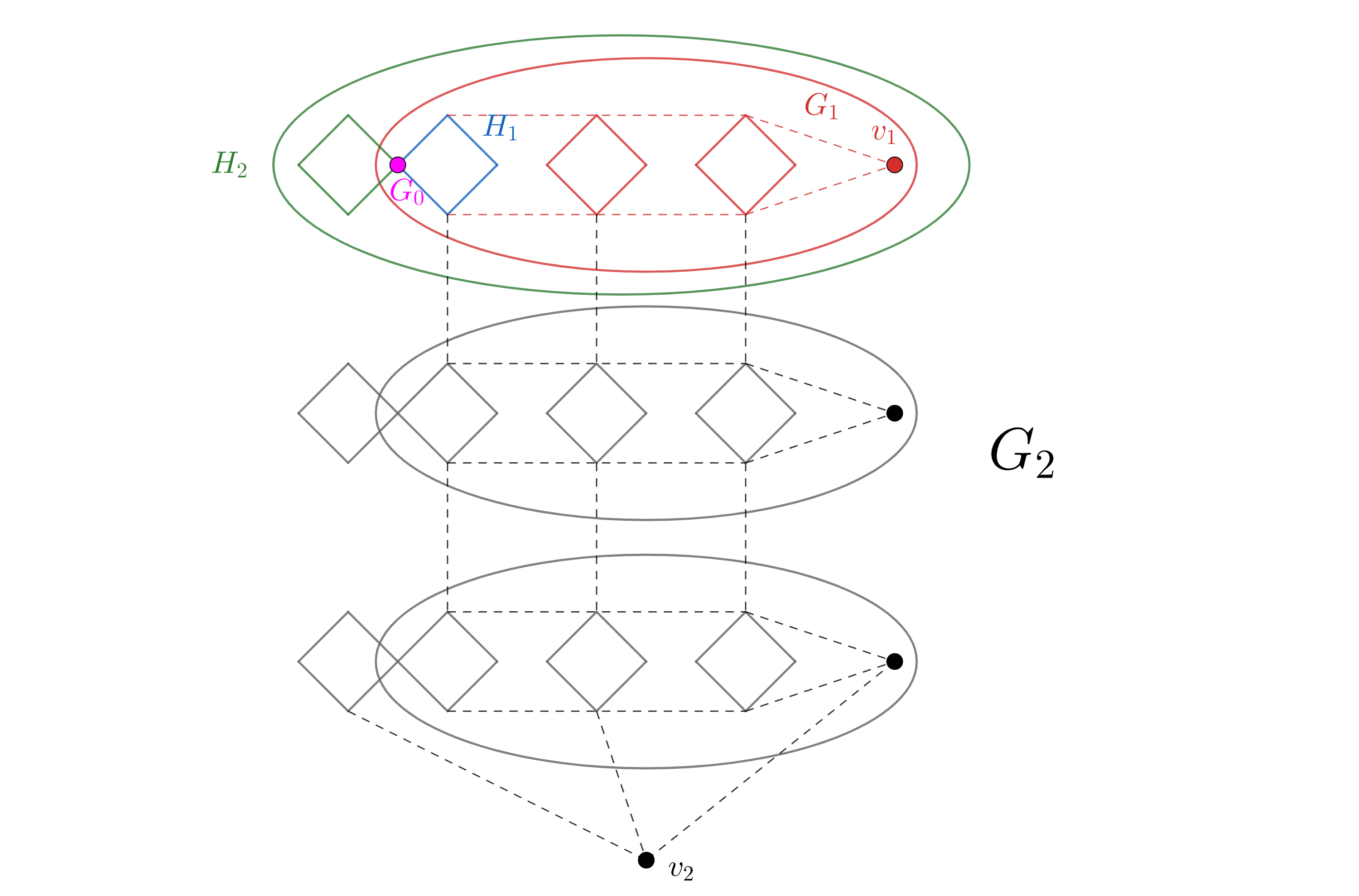}
\caption{The graph $G_2$ showing $G_0,H_1,G_1,H_2$ as subgraphs. (Note that to keep the picture manageable we have set $l_1=l_2=3).$ }
\label{fig:hive_graph}
\end{figure}
\begin{theorem}\label{lclf}
The graph $\mathcal H$ is locally finite and locally constructible, but is not a weak cop win.
\end{theorem}
\begin{proof} Certainly the graph is locally finite, as a vertex $(n, x)$ is only adjacent to vertices $(n,y)$, $(n-1, x')$ and $(n+1,x'')$, which form a finite set as $G_{n-1}$, $G_n$ and $G_{n+1}$ are finite graphs.
\begin{claim} The graph $\mathcal H$ is locally constructible.
\end{claim}
\begin{proof} We saw above that any hive graph is constructible, and hence the graphs $G_n$ are all constructible. This, combined with the above observation (about what happens when a $G_k$ is constructible), tells us that for every $n$ the subgraph of $\mathcal H$ induced by the vertices $(m, x)$ with $m\leq n$ is constructible. Thus $\mathcal H$ is locally constructible.
\end{proof}
\begin{claim} The graph $\mathcal H$ is not a weak cop win.\end{claim}
\begin{proof}
The rough idea is that if the robber is in one of the 4-cycles, say the one that appears first in $G_m$, then the cop can force the robber out of this cycle, but in order to do so he has to go to some hive vertex of some $G_m$ with $m\ge n$, which means he is a long distance away from the robber. This gives the robber time to go to the origin and back out further than the cop. 

However, as stated this is not correct: the cop can force the robber out of a 4-cycle by going to any of the copies of a hive vertex in later hive constructions, and these vertices can be arbitrarily far from the origin. Therefore, instead of looking at the cop's position itself, we look at how it `projects' onto $G_m$. To make this idea precise we need a better understanding of the hive graphs and their properties. 

Since we are dealing with several different graphs, many of which have vertices in common, in what follows, for a graph $G$, we denote by $d_G(x,y)$ and $d_G(z, A)$ the distance in $G$ between two vertices $x$ and $y$, and between a vertex $z$ and a set of vertices $A$ respectively.

Let $H$ be a finite graph and $H'$ a hive graph of $H$ with hive vertex $h$. We define the \textit{hive map} to be the function from $H'\setminus\{h\}$ to $H$ that projects the vertices to the base layer -- in other words $(x,m)$ is mapped to $(x,0)$ (and we view $(x,0)$ as identified with $x$). We note that the hive map is a graph homomorphism (but is not defined for the hive vertex).

\par Returning to the graphs $G_n$ used in the construction of $\mathcal H$, we define the \textit{one-step projection} $Q_n$ to be the function mapping $G_{n}\setminus\{v_n\}$ to $G_{n-1}$ by first applying the hive map $G_n\setminus\{v_n\}\to H_n=G_{n-1}\cup C_4$, followed by the map $G_{n-1}\cup C_4\to G_{n-1}$ that sends all the vertices in the $C_4$ to $0$. It is easy to see that the one-step projection $Q_n$ on $G_n\setminus\{v_n\}$ is a graph homomorphism.

We inductively define the \textit{$n$-projection} $J_n$ to be the map $\mathcal H\to G_n\cup\{v_k:k>n\}$  such that:
\begin{center}
$J_n((m,x))=\begin{cases}x&\text{if } m\leq n,\\
(m,x)&\text{if } m>n \text{ and } (m,x) \text{ is a hive vertex,}\\J_n((m',x'))& \text{otherwise, where } (m',x')\text{ is the one-step projection of }(m,x).\end{cases}$
\end{center}
It is important to note that the map $J_n$ is almost a graph homomophism, in the sense that it only fails to be a homomorpism for vertices that reach a hive vertex in the definition; in other words $J_n$ restricted to $J_n^{-1}(G_n)$ is a graph homomorphism. With this in mind, we classify the exceptional vertices, calling the vertices in $J_n^{-1}(v_n)$, \emph{hive-type vertices of order $n$}. Note that if $J_n(x)=v_n$ then $J_m(x)=v_n$ for all $m\le n$.

The map $J_n$ is a projection onto $G_n$. At other points in the proof we will want a projection onto $H_n$ instead of $G_n$, so we define $J'_n:\mathcal{H}\to H_n\cup \{v_k:k\ge n\}$ to be $J_n$ followed by the hive map.

\begin{lemma}\label{l:to-spine} Let $x$ be a hive-type vertex of order $n$ and $S$ the spine. Then  $d_{\mathcal H}(x,S)\ge l_n+1$.
\end{lemma}
\begin{proof} 
Fix a path from $x$ to $S$. Let $y$ be a vertex on the path $P$ of maximum hive-type order, and suppose it has order $m$. Since $x$ itself has hive-type order $n$ we see $m\ge n$. By our choice of $m$ the path $P$ is in $J_m^{-1}(G_m)$, so $J_m(P)$ is a path in $G_m$. Since any hive vertex of order $m$ maps to $v_m$ under $J_m$, and any vertex on the spine maps to $0$ under $J_m$, we see that the path $J_m(P)$ contains both $v_m$ and $0$. However, it is easy to see that $d_{G_m}(v_m,0)=l_m+1\ge l_n+1$, as the `level' in the hive graph can decrease by at most 1 at each step. The result follows.
\end{proof}
\begin{lemma}\label{l:project} Let $x$ be a hive-type vertex of order $n$ and suppose $P$ is a path of length at most $l_n$ not containing any hive-type vertex of order greater than $n$. Then $J'_{n+1}(P)$ does not contain the vertex $0$ or any vertex of the $C_4$ first added in $H_{n+1}$. 
\end{lemma}
\begin{proof} 
Since $P$ does not contain any hive-type vertex of order greater than $n$, the projection map $J'_{n+1}$ is a graph homomorphism on $P$: that is, $J'_{n+1}(P)$ is a path in $H_{n+1}$. Using Lemma~\ref{l:to-spine} (or directly), we see that $d_{H_{n+1}}(v_n,0)\ge d_\mathcal{H}(v_n,S)\ge l_n+1$.
The path starts at $v_n$ so the result follows.
\end{proof}

We are now in a position to define the robber's strategy. As in the previous construction, we have several stages of this strategy that we cycle through. In other words, the strategy allows the game to move through the different stages, or eventually remain in Stage 1. Below we explain what the stages are and, given the fact that the cop and robber are in a particular stage, how the robber can force the game into a different stage (or not leave Stage 1). We view each turn as being the cop moving followed by the robber responding.\\
\par\textit{Stage 1.} The robber is at vertex $y$ in the cycle $C_4$ that first appears in $H_m$, the cop is at vertex $x$, and $d_{H_m}(J'_m(x),y)\ge 2$. The cop moves to a vertex $x'$. If $x'$ is a hive-type vertex of order at least $m$ we move to Stage 2. Otherwise, we see that $J'_m(x)$ and $J'_m(x')$ are neighbours in $H_m$. The robber stays on the cycle $C_4$ that first appears in $H_m$, moving to a vertex $y'$ with $d_{H_m}(J'_m(x'),y')\ge 2$. In particular, the robber is not caught, and we remain in Stage 1.\\
\par \textit{Stage 2.} The robber is at vertex $y$ in the cycle $C_4$ that first appears in $H_m$, or at a point on the spine $(0,l)$  with $l<m$, and  the cop is at a hive-type vertex of order $k\ge m$.  The robber now goes to the spine in at most two steps, then to the origin in a further $m$ steps. When the robber reaches the origin we move to stage 3. By Lemma~\ref{l:to-spine} the cop's distance from the spine is at least $l_k+1>m+2$, so the robber is not caught during this stage.\\
\par\textit{Stage 3.} The robber is at the origin. Let $k'$ be the maximal order of any hive-type vertex the cop visited during Stage 2. Since at the start of Stage 2 the cop was at a hive vertex of order $k$, we have 
$k'\ge k$. The robber sets off for the vertex $(k'+1,0)$ in $H_{k'+1}$, and then to the point opposite the spine in the $C_4$ added at stage $k'+1$. This would take time $k'+1+2$. However, if at any point during this the cop reaches a hive vertex of order at least $k'+1$, the robber immediately switches back to Stage 2.\par If the cop does not go to any such vertex then let $P$ be the path followed by the cop during Stages 2 and 3. Since Stages 2 and 3 together take at most time $m+2+k'+1+2\le 2k'+5\le l_{k'}$, the path $P$ has length at most $l_{k'}$. Thus, since $P$ does not contain any hive vertex of order greater than $k$, Lemma~\ref{l:project} implies that $J_{k'+1}(P)$ does not contain 0 or any vertex of the $C_4$ first added in $H_{k'+1}$. This shows that the robber is not caught, and that this stage finishes with the robber at vertex $y$ and the cop at vertex $x$ with $d_{H_{k'+1}}(J_{k'+1}(x),y)\ge 2$, and we move back to Stage 1.\\
\par The game starts by the cop picking a vertex $y$, then the robber chooses a vertex satisfying the conditions for Stage 1. In the above strategy either the robber stays in Stage 1 after some time, which means the robber eventually stays in the same 4-cycle forever, or Stage 2 occurs infinitely often, which means the robber visits the origin infinitely often. We conclude that the graph $\mathcal H$ is not weak cop win.
\end{proof}
This concludes the proof of Theorem~\ref{lclf}. 
\end{proof}
\section{The construction time of constructible graphs}

In this section we turn to the possible ranks of a constructible graph. Recall that
the rank (or construction time) of a constructible graph $G$ is the least order-type of any construction ordering for $G$. It is easy to find graphs with construction time $n$, where $n$ is any positive integer--for example a path with $n$ vertices. By taking an infinite path we can also achieve construction time $\omega$. \par The next step is to ask if there exists a graph with construction time $\omega+1$ -- in other words we need to make infinitely many extensions and then one more at the end to be able to finish the construction. 

This was achieved by Evron, Solomon and Stahl \cite{ESS}. In fact, they showed that
the set of construction times (of countable graphs) is unbounded in the countable
ordinals. They asked, more generally, which countable ordinals can be the construction time of a graph? In this section we answer this question by constructing a graph with construction time $\gamma$, where $\gamma$ is any countable ordinal.

\par We start by giving a graph of rank $\omega+1$. We mention that this result will
be contained in our general result below (and in that general result the construction will actually be slightly different) -- we include it here to illustrate in a simpler
setting how the graph $K$ can be used.

We define the graph $G$ as follows. We take countably infinitely many disjoint copies of $K$, say $K_i$ for each positive integer $i$, and two additional vertices which we call $A$ and $B$. Let $x_i$ and $y_i$ be the vertices of $K_i$ corresponding to $x$ and $y$ in $K$. We join $A$ to $x_i$ and $y_i$ for every $i$, and $B$ to $x_i$ for every $i$. We also join $A$ and $B$. The graph $G$ is pictured below.\\
\begin{figure}[h]
    \centering
    \includegraphics[width=10cm]{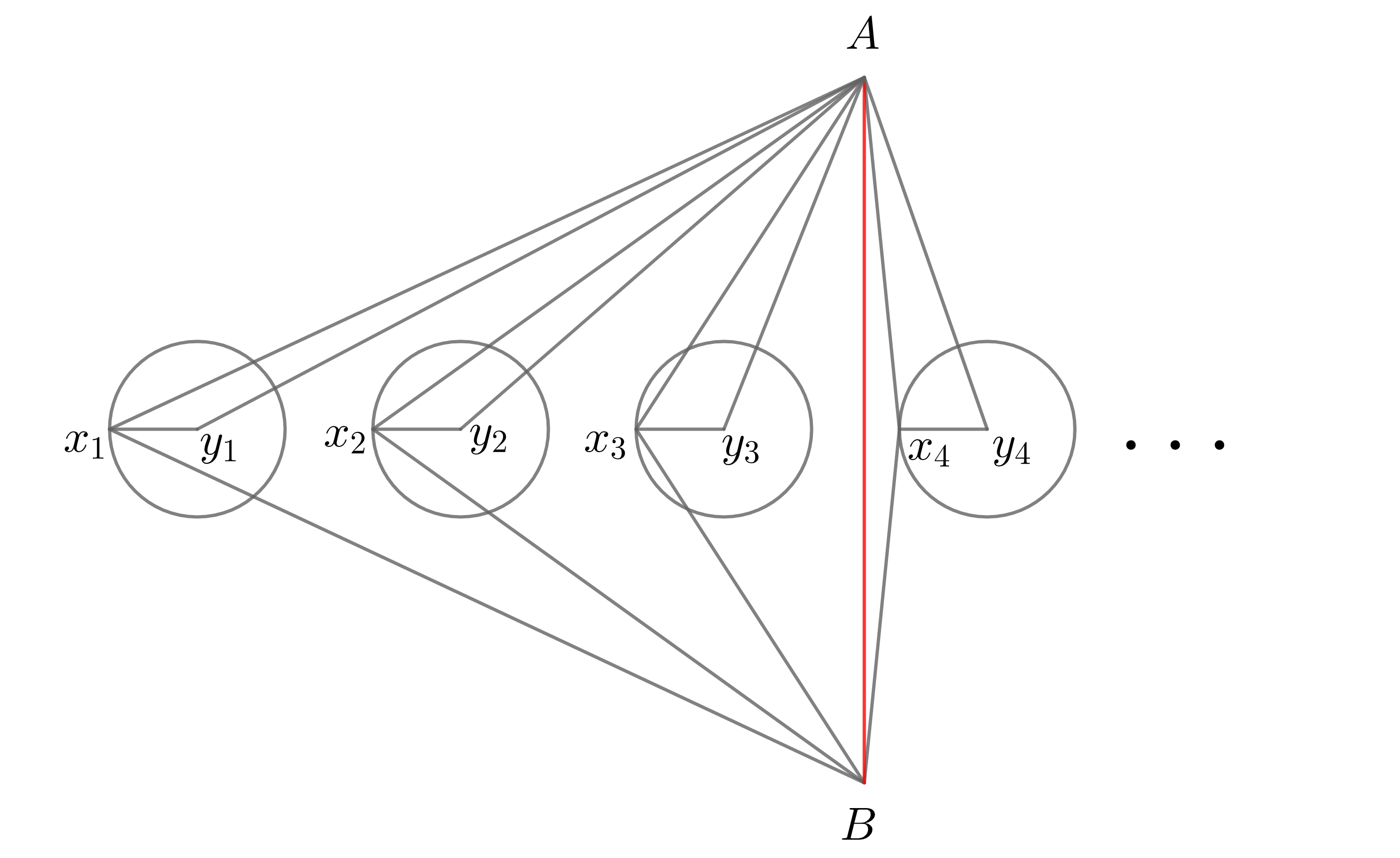}
    \caption{The graph $G$ for Theorem~\ref{t:omega_plus_1}.}
    \label{fig:omega_plus_1}
\end{figure}
\begin{theorem}\label{t:omega_plus_1}
The construction time of $G$ is $\omega+1$.
\label{omega}
\end{theorem}
\begin{proof}
To see that $G$ is constructible in time $\omega+1$ we begin with $A$, then add each copy of $K$ in turn in the following construction order: first $y$, then $z$ and $z'$
(both with parent $y$), then $w$ and $t$ with parent $z$ and $t'$ with parent $z'$, and finally $x$ with parent $y$. This is valid even with $A$ already present, since $x$ has
parent $y$ and both of these vertices in a copy of $K$ are joined to $A$.\\Finally, after doing all the above we add $B$ with parent $A$:  this is allowed since all neighbours of $B$ are also neighbours of $A$ (and $B$ and $A$ are adjacent).\par On the other hand, we cannot construct $G$ in time $\omega$. Indeed, suppose for a contradiction that there is a way to construct the graph in time $\omega$. This implies that the vertex $B$ must be constructed at some time $t$, where $t$ is a natural number. Since $t$ is finite, at time $t$ we must have some copy of $K$ with no vertices constructed yet. Let $K_i$ be such a copy. By Lemma \ref{help}, $x_i$ must be the last vertex in $K_i$ to be added, and its parent must be $y_i$. But this is impossible since $B$ is already present, and $B$ is a neighbour of $x_i$ while $y_i$ is not.
\end{proof}
\par The above result tells us that any ordinal less or equal than $\omega+1$ can be the construction time of some graph. We now prove that any countable ordinal can be achieved. We need the following simple lemma.
\begin{lemma}
Let $\alpha$ be a (non-zero) countable limit ordinal. Then there are pairwise disjoint subsets $S_i$ of $\alpha$ of order type $\alpha_i$ for all $i$, where $\alpha_1, \alpha_2\cdots$ are the ordinals less than $\alpha$.
\label{omegabeta}
\end{lemma}
\begin{proof}
Since $\alpha$ is a limit, we know that $\alpha=\omega\cdot\beta$ for some ordinal $\beta$. Since $\omega$ contains infinitely many disjoint copies of itself, it follows that $\omega\cdot\beta$ contains infinitely many disjoint copies of $\omega\cdot\beta$ too. We conclude that $\alpha$ contains infinitely many disjoint sets of order type $\alpha$.  
Let $Q_1, Q_2, \cdots$ be such a collection. Since $Q_i$ has order type $\alpha$, it has an initial segment $S_i$ of order type $\alpha_i$, as required.
\end{proof}

The following is the key result.
\begin{lemma}
Let $\lambda$ be an ordinal of the form $\lambda=\alpha+6n+1$ where $n$ is a non-negative integer and $\alpha$ is a (possibly zero) countable limit ordinal. Then there exists a constructible graph $G_{\lambda}$ with construction time $\lambda$. Moreover, $G_{\lambda}$ has two vertices, $A_{\lambda}$ and $B_{\lambda}$, such that in any construction order $B_{\lambda}$ must be added last, and there exists a construction order of time $\lambda$ that starts with $A_{\lambda}$. Furthermore, $A_{\lambda}$ is joined to $B_{\lambda}$ and, provided $n\ge 1$, $B_{\lambda}$ is not dominated by $A_{\lambda}$.
\label{ordinals}
\end{lemma}
\begin{proof} We proceed by induction. First we note that if we have found $G_{\alpha+6n+1}$ with the above properties, except possibly for the condition that $A_{\alpha+6n+1}$ does not dominate $B_{\alpha+6n+1},$ then we can find such graph for $\alpha+6(n+1)+1$ by adding a disjoint copy of $K$, identifying $B_{\alpha+6n+1}$ with the $y$ of this copy and joining $x$ to $A_{\lambda+6n+1}$. We set $B_{\alpha+6(n+1)+1}$ to be the $x$ of the $K$ copy and $A_{\alpha+6(n+1)+1}=A_{\alpha+6n+1}$. By using Lemma \ref{help} it is easy to check that all properties are satisfied. Moreover, by Lemma \ref{help} again, $B_{\alpha+6(n+1)+1}$ is not dominated by $A_{\alpha+6(n+1)+1}$.

To start with, the one-point graph satisfies the conditions for $\lambda=1$.
So to finish the proof we have to show that such graphs exist for all $\lambda=\alpha+1$ where $\alpha$ is a countable non-zero limit ordinal. So let $\alpha\geq\omega$ be a countable non-zero limit.

By induction we may assume that such graphs exist for all ordinals $\beta<\lambda=\alpha+1$ of the form $\beta = \gamma+6m+1$, where $\gamma$ is a limit ordinal and $m\ge 1$ is a positive integer. To obtain $G_{\lambda}$ we take a copy of each $G_{\beta}$ and identify all the points $A_{\beta}$ to a single vertex, which is our new $A_{\lambda}$. We also add a new vertex $B_{\lambda}$ which we join to $A_{\lambda}$ and all the vertices $B_{\beta}$. 

To see that $B_{\lambda}$ has to come last in any construction ordering, suppose that $v\not=B_\lambda$ is the last vertex added. By the induction hypothesis, $v$ has to be one of the vertices $B_{\delta}$ for $\delta<\lambda$. This vertex must be dominated by a neighbour of $B_{\lambda}$ (since they are joined), or by $B_{\lambda}$ itself. The other $B_{\beta}$ vertices are not joined to $B_{\delta}$ so they cannot dominate it. Also, by the induction hypothesis we know that $A_{\delta}$ does not dominate $B_{\delta}$, and so$A_{\lambda}$ does not dominate $B_{{\delta}}$ either. Finally, since the neighbours of $B_\lambda$ are a subset of the neighbours of $A_\lambda$, $B_\lambda$  cannot dominate $B_{\delta}$. This is a contradiction. So indeed  $B_{\lambda}$ must come last.

It is clear that the construction time of $G_\lambda\setminus\{B_{\lambda}\}$ is at least $\alpha$ because, when a $B_{\beta}$ for some $\beta<\alpha$ is added, the entire $G_{\beta}$ has to be constructed, which must take time at least $\beta$. Since $B_\lambda$ comes at the very end, $G_\lambda$ has construction time at least $\alpha+1=\lambda$.

To see that $G_{\lambda}\setminus\{B_{\lambda}\}$ has construction time at most 
$\alpha$, we use Lemma \ref{omegabeta}. Indeed, let $S_i$ be disjoint subsets of $\alpha$ of order type $\alpha_i$. We view the union of the $S_i$ as our (well-ordered) set of construction times, and at each time in $S_i$ we construct the corresponding vertex of $G_{\alpha_i}$.
 This gives a construction of time at most $\alpha$. Adding $B_\lambda$ after this takes one more step. We conclude that $G_{\lambda}$ has indeed construction time at most $\lambda$,as required. The other properties are straightforward to check.
\end{proof}
We remark that, alternatively, we could have started the induction at $\omega+1$, using the graph in Theorem~\ref{omega}.

\par We are now ready to prove the following.
\begin{theorem} For every countable ordinal $\lambda>0$, there exists a constructible graph with construction time $\lambda$.
\end{theorem}
\begin{proof} We know that for every positive integer $n$ a finite path with $n$ vertices, or indeed any constructible graph on $n$ vertices, gives us construction time $n$, and an infinite ray gives us $\omega$. From Lemma \ref{ordinals} we have that such graphs exist for all the ordinals of the form $\alpha+6n+1$ where $\alpha$ is a countable non-zero limit ordinal and $n$ is a non-negative integer. Therefore we are left to show that such graphs exist for non-zero countable limit ordinals and for ordinals of the form $\alpha+6n+i$ where $i\in\{2,3,4,5,6\}$ and $\alpha$ is a countable non-zero limit.

Suppose $\lambda=\alpha+6n+i$ where $\alpha$ is a countable non-zero limit, $n$ a non-negative integer and $2 \leq i \leq 6$. We take the graph $G_{\alpha+6n+1}$ constructed in Lemma \ref{ordinals} and add say a path of $i-1$ vertices attached to the vertex $B_{\alpha+6n+1}$.

Finally, suppose $\lambda$ is a countable non-zero limit ordinal. In this case we note that the graph $G_{\lambda+1}\setminus\{B_{\lambda+1}\}$ is a suitable choice.
\end{proof}

\section{Open Problems}
The obvious open problem is to classify which graphs are weak-cop wins.
\begin{question}
Which graphs are weak cop wins?
\end{question}
There is also the question of which graphs are actual cop wins. However,
as there are so many constructible graphs that are not cop wins (e.g. $\mathbb{Z}$), 
and as we have seen there is a graph that is a cop win and not constructible, 
a structural classification is very open.

There are also even weaker notions of win that we could consider: for example, we could view it as a win for the cop if he can force the robber to leave (but possibly return to) any finite set.
\begin{question}
Which graphs have the property that the cop has a strategy that ensures that, given any finite set of vertices, the robber must leave this set at some point (although he may return to this set later), or get caught?  
\end{question}
Note that if there is a such a strategy for
each individual finite set then, by concatenating these (necessarily finite time)
strategies, we do 
obtain a single strategy that works for all finite sets.
Obviously any locally constructible graph has this property, but we do not know whether the converse holds. The example of a graph that is locally constructible but not a weak cop win does show that this is strictly weaker notion than that of a weak cop win.

Finally, we have seen that there are graphs where the robber can avoid being trapped in one end of the graph (recall the doubly infinite chain of copies of $K$ described at
the end of Section 3). In particular, that graph is a weak cop win in which the robber can return to a specified vertex an arbitrarily long time after he first visited it. However, we do not know the answer to the following question.
\begin{question}
Is there a graph $G$ which is a weak cop win but such that the robber can guarantee to revisit his initial vertex $v$ after an arbitrarily long time, and then guarantee to
revisit $v$ again after another arbitrarily long time?
\end{question}
More precisely, for each cop starting position the robber has a starting position
$v$ such that, for every pair of positive integers $m$ and $n$, the robber has a strategy that ensures that he does not get caught
and either he stays in some finite set forever or he returns to $v$ at some time $t \geq m$ and also at some time $s \geq t+n$.

\Addresses
\end{document}